\documentclass[10pt]{amsart}

\newtheorem{theorem}{Theorem}
\newtheorem{lemma}{Lemma}

\newtheorem{cor}{Corollary}
\newtheorem{prop}{Proposition}

\usepackage{amsmath}
\usepackage[psamsfonts]{amssymb}
\usepackage[mathscr]{euscript}

\newcommand{\cal}{\mathcal}

\title[Product twistor spaces and Weyl geometry] {Product twistor spaces and Weyl geometry}
\author{Johann Davidov}

\thanks{The author is partially supported by  the National Science
Fund, Ministry of Education and Science of Bulgaria under contract
DN 12/2}

\address{Institute of Mathematics and Informatics \\
Bulgarian Academy of Sciences\\ Acad. G.Bonchev Str. Bl.8\\
1113 Sofia\\ Bulgaria} \email{jtd@math.bas.bg}

\begin{document}

\begin{abstract}

Motivated by generalized geometry (\`a la Hitchin), we discuss the
integrability conditions for four natural almost complex structures
on the pro-duct bundle ${\mathcal Z}\times {\mathcal Z}\to M$, where
${\mathcal Z}$ is the twistor space of a Riemannian 4-manifold $M$
endowed with a metric connection $D$ with skew-symmetric torsion.
These structures are defined by means of the connection $D$ and four
(K\"ahler) complex structures on the fibres of this bundle. Their
integrability conditions are interpreted in terms of Weyl geometry
and this is used to supply examples satisfying the conditions.

\vspace{0,1cm} \noindent 2010 {\it Mathematics Subject Classification} 53C28; 53C15, 53D18.

\vspace{0,1cm} \noindent {\it Key words: twistor spaces, almost
complex structures, skew-symmetric torsion, Weyl geometry}.

\end{abstract}

\maketitle 

\hfill Dedicated to the memory of Thomas Friedrich

\vspace{0.5cm}

\section{Introduction}

The motivation for this paper comes from generalized geometry in the
sense of N. Hitchin \cite{Hit02}.

Recall that a generalized metric \cite{Gu, Witt} on a smooth
manifold $M$ is a subbundle $E$ of $TM\oplus T^{\ast}M$ such that
the restriction to $E$ of the metric
$<X+\alpha,Y+\beta>=\alpha(Y)+\beta(X)$ is positive definite,
$X,Y\in TM$, $\alpha,\beta\in T^{\ast}M$. Every generalized metric
is determined by a positive definite metric $g$ and a skew-symmetric
$2$-form $\Theta$ on $M$ such that
$E=\{X+\imath_{X}g+\imath_{X}\Theta: X\in TM\}$. Moreover, $E$
determines a connection on $M$ which is metric with respect to $g$
and whose torsion form is skew-symmetric, equal to $d\Theta$, see
\cite{Hit10}.

Given a generalized metric $E$, denote by ${\mathcal G}(E)$ the
bundle over $M$ whose fibre at a point $p\in M$ consists of
generalized complex structures $J$ on the tangent space $T_pM$
compatible with the generalized metric $E_p$, the fibre of $E$ at
$p$, in the sense that $JE_p\subset E_p$. The bundle ${\mathcal
G}(E)$, called the generalized twistor space of $(M,E)$ in
\cite{Dav}, admits a generalized almost complex structure ${\mathcal
J}$ defined by means of the connection determined by $E$. This
structure resembles the Atiyah-Hitchin-Singer  almost complex
structure on the usual twistor space \cite{AHS}. The integrability
condition for ${\mathcal J}$ (in the sense of generalized geometry)
has been found in \cite{Des} in the case when $dim\,M=4$ and
$\Theta=0$, and in \cite{Dav} in the general case. Besides the
generalized almost complex structure ${\mathcal J}$, the manifold
${\mathcal G}(E)$ admits four natural  almost complex structure
$\mathscr{J}^m$ in the usual sense since ${\mathcal G}(E)$ is
diffeomorphic to the product bundle ${\mathcal Z}\times {\mathcal
Z}\to M$ where ${\mathcal Z}$ is the (usual) twistor space of the
Riemannian manifold $(M,g)$, see, for example, \cite{Dav}. They are
defined by means of a metric connection $D$ on $M$ with
skew-symmetric torsion and four complex structures (in fact,
K\"ahler ones) on the fibres of ${\mathcal Z}\times {\mathcal Z}$,
and are not products of almost complex structures on ${\mathcal Z}$.
Two of these structures, denoted by ${\mathscr J}^1$ and ${\mathscr
J}^2$, can be considered as analogs of  the Atiyah-Hitchin-Singer
almost complex structure, while ${\mathscr J}^3$ and ${\mathscr
J}^4$ correspond to the Eells-Salamon almost complex structure on
${\mathcal Z}$ \cite{ES}. In the this paper, we deal with the
classical integrability problem for these structures. The
integrability condition for the almost complex structure ${\mathscr
J}^1$ has been found in \cite{Des} when $dim\,M=4$ and $\Theta=0$
(so $D$ is the Levi-Civita connection of $g$). In the present paper,
we also consider the most interesting case of an oriented
four-dimensional manifold, but the connection used here has
skew-symmetric torsion. In order to find the corresponding
integrability condition, we briefly discuss the
Atiyah-Hitchin-Singer almost complex structure on ${\mathcal Z}$
defined by means of a metric connection with skew-symmetric torsion.
The integrability condition for this structure does not depend on
the torsion and is (anti-) self-duality of the metric. This,
suggested by certain facts, formally follows easily from a result in
\cite{D-V}. It is a simple observation that the almost complex
structures ${\mathscr J}^3$ and ${\mathscr J}^4$ are never
integrable. On the other hand, the integrability conditions for the
restrictions of $\mathscr{J}^1$ and ${\mathscr J}^2$ to the
connected components of ${\mathcal Z}\times{\mathcal Z}$ involve
(anti-) self-duality of the metric and certain relations between the
Ricci curvature of the Levi-Civita connection and the torsion of the
given connection. The latter are interpreted in terms of Weyl
geometry and this is used to supply examples of manifolds satisfying
the integrability conditions.

\section{Preliminaries}

Let $(M,g)$ be an oriented  Riemannian manifold of dimension four.
The metric $g$ induces a metric on the bundle of two-vectors
$\pi:\Lambda^2TM\to M$ by the formula
$$
g(v_1\wedge v_2,v_3\wedge v_4)=\frac{1}{2}det[g(v_i,v_j)],
$$
the  factor $1/2$ being chosen in consistence with \cite{DM89}.

Let $\ast:\Lambda^kTM\to \Lambda^{4-k}M$, $k=0,...,4$, be the Hodge
star operator. Its restriction to $\Lambda^2TM$ is an involution,
thus we have the orthogonal decomposition
$$
\Lambda^2TM=\Lambda^2_{-}TM\oplus\Lambda^2_{+}TM,
$$
where $\Lambda^2_{\pm}TM$ are the subbundles of $\Lambda^2TM$
corresponding to the $(\pm 1)$-eigenvalues of the operator $\ast$

\smallskip

Let $(E_1,E_2,E_3,E_4)$ be a local oriented orthonormal frame of
$TM$. Set
\begin{equation}\label{s-basis}
s_1^{\pm}=E_1\wedge E_2\pm E_3\wedge E_4, \quad s_2^{\pm}=E_1\wedge
E_3\pm E_4\wedge E_2, \quad s_3^{\pm}=E_1\wedge E_4\pm E_2\wedge
E_3.
\end{equation}
Then $(s_1^{\pm},s_2^{\pm},s_3^{\pm})$ is a local orthonormal frame
of $\Lambda^2_{\pm}TM$. This frame defines an orientation on
$\Lambda^2_{\pm}TM$ which does not depend on the choice of the frame
$(E_1,E_2,E_3,E_4)$ (see, for example, \cite{D17}). We call this
orientation "canonical".

For every $a\in\Lambda ^2TM$, define a skew-symmetric endomorphism
of $T_{\pi(a)}M$ by

\begin{equation}\label{cs}
g(K_{a}X,Y)=2g(a, X\wedge Y), \quad X,Y\in T_{\pi(a)}M.
\end{equation}
Denoting by $G$ the standard metric $-\frac{1}{2}Trace\,PQ$ on the
space of skew-symmetric endomorphisms, we have $G(K_a,K_b)=2g(a,b)$
for $a,b\in \Lambda ^2TM$. If $a\in\Lambda^2TM$ is of unit length,
then $K_a$ is a complex structure on the vector space $T_{\pi(a)}M$
compatible with the metric $g$, i.e., $g$-orthogonal. Conversely,
the $2$-vector $a$ dual to one half of the fundamental $2$-form of
such a complex structure is a unit vector in $\Lambda^2TM$.
Therefore the unit sphere bundle ${\mathcal Z}$ of $\Lambda ^2TM$
parametrizes the complex structures on the tangent spaces of $M$
compatible with the metric $g$. This bundle is called the twistor
space of the Riemannian manifold $(M,g)$. Since $M$ is oriented, the
manifold ${\mathcal Z}$ has two connected components ${\mathcal
Z}_{\pm}$ called the positive and the negative twistor spaces of
$(M,g)$. These are the unit sphere subbundles of
$\Lambda^2_{\pm}TM$. The bundle ${\mathcal Z}_{\pm}\to M$
parametrizes  the complex structures on the tangent spaces of $M$
compatible with the metric and  $\pm$ the orientation via the
correspondence ${\mathcal Z}_{\pm}\ni\sigma\to K_{\sigma}$.

\smallskip

The vertical space ${\mathcal V}_{\sigma}=\{V\in T_{\sigma}{\mathcal
Z}_{\pm}:~ \pi_{\ast}V=0\}$ of the bundle $\pi:{\mathcal Z}_{\pm}\to
M$ at a point $\sigma$ is the tangent space to the fibre of
${\mathcal Z}_{\pm}$ through $\sigma$. Thus, considering
$T_{\sigma}{\mathcal Z}_{\pm}$ as a subspace of
$T_{\sigma}(\Lambda^2_{\pm}TM)$, ${\mathcal V}_{\sigma}$ is the
orthogonal complement of ${\Bbb R}\sigma$ in
$\Lambda^2_{\pm}T_{\pi(\sigma)}M$.


Let $D$ be a metric connection on $(M,g)$. The induced connection on
$\Lambda^2TM$ will also be denoted by $D$. If
$(s_1^{\pm},s_2^{\pm},s_3^{\pm})$ is the orthonormall frame of
$\Lambda ^2_{\pm}TM$ defined by means of an oriented orthonormal
frame $(E_1,...,E_4)$ of $TM$ via (\ref{s-basis}), we have
$g(D_{X}s_i^{+},s_j^{-})= g(D_{X}s_i^{-},s_j^{+})=0$ and
$g(D_{X}s_i^{\pm},s_j^{\pm})=-g(D_{X}s_j^{\pm},s_i^{\pm})$ for every
$X\in TM$ and every $i,j=1,2,3$. Therefore the connection $D$
preserves the bundles $\Lambda^2_{\pm}TM$ and induces a metric
connection on each of these bundles denoted again by $D$. Let
$\sigma\in{\mathcal Z}_{\pm}$,  and let $s$ be a local section of
${\mathcal Z}_{\pm}$ such that $s(p)=\sigma$ where $p=\pi(\sigma)$.
Considering $s$ as a section of $\Lambda^2_{\pm}TM$, we have
$D_{X}s\perp s(p)$, i.e., $D_{X}s\in{\cal V}_{\sigma}$ for every
$X\in T_pM$ since $s$ has a constant length. Moreover,
$X^h_{\sigma}=s_{\ast}X-D_{X}s\in T_{\sigma}{\mathcal Z}_{\pm}$ is
the horizontal lift of $X$ at ${\sigma}$ with respect the connection
$D$ on $\Lambda^2_{\pm}TM$. Thus, the horizontal distribution of
$\Lambda^2_{\pm}TM$ with respect to $D$ is tangent to the twistor
space ${\mathcal Z}_{\pm}$. In this way, we have the decomposition
$T{\mathcal Z}_{\pm}={\mathcal H}\oplus {\mathcal V}$ of the tangent
bundle of ${\mathcal Z}_{\pm}$ into horizontal and vertical
components. The horizontal and vertical parts of a tangent vector
$A\in T{\mathcal Z}_{\pm}$ will be denoted by ${\mathcal H}A$ and
${\mathcal V}A$, respectively.

For $\sigma\in{\mathcal Z}_{\pm}$, the horizontal space ${\mathcal
H}_{\sigma}$ is isomorphic to the tangent space $T_{\pi(\sigma)}M$
via the differential $\pi_{\ast\,\sigma}$. The vertical space
${\mathcal V}_{\sigma}$ is tangent to the unit sphere in the
$3$-dimensional vector space $(\Lambda^2_{\pm}T_{\pi(\sigma)}M,g)$,
and we denote by ${\mathcal J}_{\sigma}$ the standard complex
structure of the unit sphere restricted to ${\mathcal V}_{\sigma}$.
It is given by
$$
{\mathcal J}_{\sigma}V=\pm(\sigma\times V),\quad V\in{\mathcal
V}_{\sigma},
$$
where $\times$ is the usual vector-cross product on the
$3$-dimensional Euclidean space
$(\Lambda^2_{\pm}T_{\pi(\sigma)}M,g)$ endowed with its canonical
orientation.

\smallskip

\noindent {\bf Convention}. In what follow, we shall freely identify
the bundle $\Lambda ^2TM$ with the bundle $A(TM)$ of
$g$-skew-symmetric endomorphism of $TM$ by means of the isomorphism
$a\to K_a$ defined by (\ref{cs}).

\smallskip

Using the basis (\ref{s-basis}), it is easy to check that if
$a,b\in\Lambda^2_{\pm}T_pM$, the isomorphism $\Lambda^2TM\cong
A(TM)$ sends $a\times b$ to $\pm\frac{1}{2}[K_a,K_b]$. In the case
when $a\in\Lambda^2_{+}T_pM$, $b\in\Lambda^2_{-}T_pM$, the
endomorphisms $K_a$ and $K_b$ of $T_pM$ commute. If
$a,b\in\Lambda_{\pm}T_pM$,
$$
K_{a}\circ K_{b}=-g(a,b)Id\pm K_{a\times b}.
$$
In particular, $K_a$ and $K_b$, $a,b\in\Lambda_{\pm}T_pM$,
anti-commute if and only if $a$ and $b$ are orthogonal.

\smallskip

The $6$-manifold ${\mathcal Z}_{\pm}$ admits two almost complex
structures ${\cal J}_1$ and ${\cal J}_2$ introduced, respectively,
by Atiyah-Hitchin-Singer \cite{AHS}, and Eells-Salamon \cite{ES} in
the case when the connection $D$ is the Levi-Civita connection of
$(M,g)$. On a horizontal space ${\mathcal H}_{\sigma}$,
$\sigma\in{\mathcal Z}_{\pm}$, the structures ${\mathcal J}_1$ and
${\mathcal J}_2$ are both defined as the lift to ${\mathcal
H}_{\sigma}$ of the complex structure $K_{\sigma}$ on
$T_{\pi(\sigma)}M$. On a vertical space ${\mathcal V}_{\sigma}$,
${\cal J}_1$ is defined to be the complex structure ${\cal
J}_{\sigma}$ of the fibre through $\sigma$, while ${\mathcal J}_2$
is defined as the conjugate complex structure, i.e., ${\mathcal
J}_{2}|{\mathcal V}_{\sigma}=-{\cal J}_{\sigma}$. Thus, for
$J\in{\mathcal Z}_{\pm}$,
\begin{equation}\label{ACS}
\begin{array}{c}
{\cal J}_{n}|{\mathcal H}_{J}=(\pi_{\ast}|{\mathcal
H}_{J})^{-1}\circ J \circ\pi_{\ast}|{\mathcal
H}_{J}\\[8pt]

{\cal J}_{n}V=(-1)^{n+1}(J\circ V )\;~~\mbox {for} ~~\; V\in
\mathcal{V_{J}},\;\;\;\;n=1,2.
\end{array}
\end{equation}

\smallskip

Suppose that $D$ is the Levi-Civita connection $\nabla$ of $(M,g)$.
By a result of Eells-Salamon \cite{ES}, the almost complex structure
${\mathcal J}_2$ is never integrable, so it does not come from a
complex structure. Nevertheless, ${\mathcal J}_2$ is very useful for
constructing harmonic maps. The integrability condition for
${\mathcal J}_1$ has been found by Atiyah-Hitchin-Singer \cite{AHS}.
To state their result, we first recall  the well-known curvature
decomposition in dimension four. For the curvature tensor $R^{D}$ of
a connection $D$ on $M$, we adopt the following definition:
$R^D(X,Y)=D_{[X,Y]}-[D_{X},D_{Y}]$. The curvature operator
${\mathcal R}^D$ corresponding to the curvature tensor is the
endomorphism of $\Lambda ^{2}TM$ defined by
$$
g({\mathcal R}^D(X\wedge Y),Z\wedge U)=g(R^D(X,Y)Z,U),\quad
X,Y,Z,U\in TM.
$$
Now, let $\nabla$ be the Levi-Civita connection of $(M,g)$. Denote
by $\rho:TM\to TM$ its Ricci operator, $g(\rho(X),Y)=Ricci(X,Y)$,
and by $s$ the scalar curvature. Then the endomorphism ${\cal
B}:\Lambda^2TM\to \Lambda^2TM$  corresponding to the traceless Ricci
tensor is given by
\begin{equation}\label{B}
{\mathcal B}(X\wedge Y)=\rho(X)\wedge
Y+X\wedge\rho(Y)-\frac{s}{2}X\wedge Y.
\end{equation}
Note that ${\mathcal B}$ sends  $\Lambda^2_{\pm}TM$ into
$\Lambda^2_{\mp}TM$. Let ${\mathcal W}: \Lambda^2TM\to \Lambda^2TM$
be the endomorphism corresponding to the Weyl conformal tensor.
Denote the restriction of ${\mathcal W}$ to $\Lambda^2_{\pm}TM$ by
${\mathcal W}_{\pm}$, so ${\mathcal W}_{\pm}$ sends
$\Lambda^2_{\pm}TM$ to $\Lambda^2_{\pm}TM$ and vanishes on
$\Lambda^2_{\mp}TM$.

It is well known that the curvature operator decomposes as
(\cite{ST}, see, e.g., \cite[Chapter 1 H]{Besse})
\begin{equation}\label{dec}
{\mathcal R}^{\nabla}=\frac{s}{6}Id+{\mathcal B}+{\mathcal
W}_{+}+{\mathcal W}_{-}
\end{equation}
Note that this differs from \cite{Besse} by  a factor $1/2$  because
of the factor $1/2$ in our definition of the induced metric on
$\Lambda^2TM$.

The Riemannian manifold $(M,g)$ is Einstein exactly when ${\mathcal
B}=0$. It is called self-dual (anti-self-dual) if ${\mathcal
W}_{-}=0$ (resp. ${\mathcal W}_{+}=0$). The self-duality
(anti-self-duality) condition is invariant under conformal changes
of the metric since the Weyl tensor is so.

Note that changing the orientation of $M$ interchanges the roles of
$\Lambda^2_{-}TM$ and $\Lambda^2_{+}TM$ (respectively, of ${\mathcal
Z}_{-}$ and ${\mathcal Z}_{+}$), hence the roles of ${\mathcal
W}_{-}$ and ${\mathcal W}_{+}$.

The famous Atiyah-Hitchin-Singer  theorem \cite{AHS} states that the
almost complex structure ${\cal J}_1$ on ${\cal Z}_{-}$ (resp.
${\cal Z}_{+}$) is integrable if and only if $(M,g)$ is self-dual
(resp. anti-self-dual).

\section{Technical lemmas}

Let $D$ be a metric connection on $(M,g)$, and consider the almost
complex structures ${\mathcal J}_n$, $n=1,2$, defined on the
positive twistor space ${\mathcal Z}_{+}$ by means of the connection
$D$. Denote by ${\mathcal N}_n$ the Nijenhuis tensor of the almost
complex structure ${\mathcal J}_n$:
$$
{\mathcal N}_n(A,B)=-[A,B]+[{\mathcal J}_nA,{\mathcal
J}_nB]-{\mathcal J}_n[{\mathcal J}_nA,B]-{\mathcal J}_n[A,{\mathcal
J}_nB].
$$
Let $T$ be the torsion tensor of the connection $D$.
\begin{lemma}\label{Nijen}\rm{(\cite{BBO,DM89,D-V,G,OR})}
Let $J\in{\mathcal Z}_{+}$, $V,W\in{\mathcal V}_J$, and let $X,Y\in
T_{\pi(J)}M$. Then
$$
\begin{array}{l}
(i)\>{\mathcal H}{\mathcal
N}_n(X^h,Y^h)_J=(T(X,Y)-T(JX,JY)+JT(JX,Y)+JT(X,JY))^h_J.\\[8pt]
(ii)\>{\mathcal V}{\mathcal N}_n(X^h,Y^h)_J=-R^D(X,Y)J+R^D(JX,JY)J\\[6pt]
\hspace{6cm}-{\mathcal J}_nR^D(JX,Y)J -{\mathcal J}_nR^D(X,JY)J.\\[8pt]
(iii)\>{\mathcal N}_n(X^h_J,V)=[(-1)^{n}+1](JVX)^h_J.\\[8pt]
(iv)\> {\mathcal N}_n(V,W)=0.
\end{array}
$$
\end{lemma}

\begin{cor}
The almost complex structure ${\mathcal J}_2$ is never integrable.
\end{cor}
\begin{proof}
If $(E_1,...,E_4)$ is an oriented orthonormal frame, set
$J=E_1\wedge E_2+E_3\wedge E_4$, $V=E_1\wedge E_3+E_4\wedge E_2$.
Then ${\mathcal V}{\mathcal N}_2(E_1^h,V)_J=2(E_4)^h_J\neq 0$.
\end{proof}

\section{Integrability of the Atiyah-Hitchin-Singer almost complex structure on the twistor spaces of a Riemannian $4$-manifold with
skew-symmetric torsion}

Suppose that the metric connection $D$ has (totally) skew-symmetric
torsion, i.e., the trilinear form
$$
{\mathcal T}(X,Y,Z)=g(T(X,Y),Z), \quad X,Y,Z\in TM,
$$
is skew-symmetric. Let $\nabla$ be the Levi-Civita connection of
$(M,g)$. Then
$$
D_{X}Y=\nabla_{X}Y+\frac{1}{2}T(X,Y).
$$
The difference between  the connections $D$ and $\nabla$ satisfies
the following identity.

\begin{lemma}\label{horN}
If the torsion $T$ is skew-symmetric, then for every $J\in{\mathcal
Z}_{+}$ and every $X,Y\in T_{\pi(J)}M$,
$$
T(X,Y)-T(JX,JY)+JT(JX,Y)+JT(X,JY)=0
$$
\end{lemma}

\begin{proof}
Take an oriented orthonormal basis $(E_1,...,E_4)$ of $T_{\pi(J)}M$
such that $E_2=JE_1$, $E_4=JE_3$. Then it is easy to check that
$$
{\mathcal T}(E_i,E_j,E_k)-{\mathcal T}(JE_i,JE_j,E_k)-{\mathcal
T}(JE_i,E_j,JE_k)-{\mathcal T}(E_i,JE_j,JE_k)=0
$$
for every $i,j,k=1,...,4$.
\end{proof}

Now, it follows from \cite[Lemma 2.4]{D-V} that the connections $D$
and $\nabla$ define the same almost complex structure on the twistor
space. Thus, by the Atiyah-Hitchin-Singer  theorem \cite{AHS}, we
have:

\begin{theorem}\label{Z+}
Let $(M,g)$ be an oriented Riemannian $4$-manifold endowed with a
metric connection with skew-symmetric torsion. The
Atiyah-Hitchin-Singer almost complex structure ${\mathcal J}_1$ on
the twistor space ${\mathcal Z}_{+}$  is integrable if and only if
the metric $g$ is anti-self-dual.
\end{theorem}

\smallskip

\noindent {\bf Remark}. This result can also be proved  by means of
Lemmas~\ref{Nijen} and \ref{horN}, and the following relation
between the curvature tensors $R^D$ and $R^{\nabla}$  of the
connections $D$ and $\nabla$:
\begin{equation}\label{curv-D}
\begin{array}{c}
g(R^D(X,Y)Z,U)=g(R^{\nabla}(X,Y)Z,U)\\[6pt]
-\displaystyle{\frac{1}{2}}\big[(\nabla_{X}{\mathcal T})(Y,Z,U)-(\nabla_{Y}{\mathcal T})(X,Z,U)\big]\\[6pt]
+\displaystyle{\frac{1}{4}}\sum\limits_{i=1}^4\big[{\mathcal
T}(X,U,E_i){\mathcal T}(Y,Z,E_i)-{\mathcal T}(X,Z,E_i){\mathcal
T}(Y,U,E_i)\big],
\end{array}
\end{equation}
where $X,Y,Z,U\in T_pM$ and $\{E_1,...,E_4\}$ is an orthonormal
basis of $T_pM$.

\section{Integrability of natural almost complex structures on a product twistor
space}

 Every metric connection $D$ on $(M,g)$ induces a metric connection
(again denoted by $D$) on the product bundle $\pi:\Lambda^2TM\times
\Lambda^2TM\to M$, which preserves the four subbundles
$\Lambda^2_{\pm}TM\times \Lambda^2_{\pm}TM$. If ${\mathcal Z}$ is
the twistor space of $(M,g)$, we can define four  almost complex
structures $\mathscr{J}^m$ on the product bundle ${\mathcal Z}\times
{\mathcal Z}\to M$ (the generalized twistor space) as follows. If
$J=(J_1,J_2)\in {\mathcal Z}\times {\mathcal Z}$, the horizontal
space ${\mathcal H}_J$ at $J$ of the vector bundle
$\Lambda^2TM\times \Lambda^2TM$ with respect to the connection $D$
is tangent to ${\mathcal Z}\times {\mathcal Z}$, and we set
$\mathscr{J}^mX^h_J=(J_1X)^h_J$ for every $X\in T_{\pi(J)}M$ and
$m=1,...,4$. There are four natural complex structures on the
vertical space ${\mathcal V}_J={\mathcal V}_{J_1}\times{\mathcal
V}_{J_2}$ of ${\mathcal Z}\times {\mathcal Z}$ defined for
$V=(V_1,V_2)$ by $K_1V=(J_1V_1,J_2V_2)$, $K_2V=(J_1V_1,-J_2V_2)$,
$K_3V=-K_2V$, $K_4V=-K_1V$. We set $\mathscr{J}^mV=K_mV$,
$m=1,...,4$.

\smallskip

\noindent {\bf Remark}. In order to define a complex structure on
the horizontal spaces ${\mathcal H}_J$, $J=(J_1,J_2)$, we can use
the structure $J_2$ instead of $J_1$. Then we get four more almost
complex structures on ${\mathcal Z}\times {\mathcal Z}$. But, by
symmetry, they do not differ essentially from the structures
$\mathscr{J}^m$, $m=1,...,4$.

\smallskip

In this section, we shall find the integrability conditions for the
restrictions of $\mathscr{J}^m$ to the four connected components of
${\mathcal Z}\times {\mathcal Z}$, the subbundles ${\mathcal
Z}_{\pm}\times {\mathcal Z}_{\pm}$, under the assumption that {\it
the torsion of $D$ is skew-symmetric}.

Denote by ${\mathscr N}_m$ the Nijenhuis tensor of the almost
complex structure $\mathscr{J}^m$. Then we have the following analog
of Lemma~\ref{Nijen} (with a similar proof).

\begin{prop}\label{Ni}
Let $J=(J_1,J_2)\in{\mathcal Z}\times{\mathcal Z}$,
$V=(V_1,V_2),W\in{\mathcal V}_J$, and let $X,Y\in T_{\pi(J)}M$. Then
$$
\begin{array}{l}
(i)\>{\mathcal H}{\mathscr
N}_m(X^h,Y^h)_J=(T(X,Y)-T(J_1X,J_1Y)+J_1T(J_1X,Y)+J_1T(X,J_1Y))^h_J.\\[8pt]
(ii)\>{\mathcal V}{\mathscr N}_m(X^h,Y^h)_J=-R^D(X,Y)J+R^D(J_1X,J_1Y)J\\[6pt]
\hspace{6cm}-K_m(R^D(J_1X,Y)J +R^D(X,J_1Y)J).\\[8pt]
(iii)\>{\mathscr N}_m(X^h_J,V)=0\quad for \quad
m=1,2.\\[6pt] {\mathscr
N}_m(X^h_J,V)=2(J_1V_1X)^h_J \quad for \quad
m=3,4.\\[8pt]
(iv)\> {\mathscr N}_m(V,W)=0.
\end{array}
$$
\end{prop}

\begin{cor}
The restrictions of the almost complex structures $\mathscr{J}^m$, $m=3,4$, to the connected components of ${\mathcal
Z}\times {\mathcal Z}$  are not integrable.
\end{cor}

\smallskip

In order to determine the integrability conditions for
$\mathscr{J}^1$  and $\mathscr{J}^2$, we need the following.

\smallskip

\noindent {\bf Notation}. The $1$-form $\ast{\mathcal T}$ will be
denoted by $\tau$.

Clearly, the form $\tau$ uniquely determines the $3$-form ${\mathcal
T}$, hence the connection $D$.

\medskip

For a given orthonormal frame $E_1,...,E_4$, it is convenient to set
$$
E_{ijk}=E_i\wedge E_j\wedge E_k, \quad {\mathcal T}_{ijk}={\mathcal
T}(E_{ijk}).
$$
Under this notation, we have
\begin{equation}\label{T-tau}
\begin{array}{c}
{\mathcal T}={\mathcal T}_{123}E_{123}+{\mathcal
T}_{124}E_{124}+{\mathcal T}_{134}E_{134}+{\mathcal
T}_{234}E_{234}\\[8pt]
\tau=-{\mathcal T}_{234}E_1+{\mathcal T}_{134}E_2-{\mathcal
T}_{124}E_3+{\mathcal T}_{123}E_4.
\end{array}
\end{equation}
It follows that
\begin{equation}\label{nabla-T-tau}
\begin{array}{c}
(\nabla_{X}{\mathcal T})(E_{123})=(\nabla_{X}\tau)(E_4),\quad
(\nabla_{X}{\mathcal T})(E_{124})=-(\nabla_{X}\tau)(E_3),\\[8pt]
(\nabla_{X}{\mathcal T})(E_{134})=(\nabla_{X}\tau)(E_2),\quad
(\nabla_{X}{\mathcal T})(E_{234})=-(\nabla_{X}\tau)(E_1).
\end{array}
\end{equation}

\bigskip
\begin{theorem}\label{Z++}
The restriction of the almost complex structure $\mathscr{J}^m$,
$m=1$ or $2$, to ${\mathcal Z}_{+}\times {\mathcal Z}_{+}$
(respectively, ${\mathcal Z}_{-}\times {\mathcal Z}_{-}$) is
integrable if and only if the metric $g$ is anti-self-dual (resp.
self-dual) with scalar curvature
$\displaystyle{s=\frac{3}{2}||\tau||^2+3\delta\tau}$ and the
$2$-form $d\tau$  is anti-self-dual (resp. self-dual).

\end{theorem}

\begin{proof}
According to Proposition~\ref{Ni} and Lemma~\ref{horN},  the almost complex structure $\mathscr{J}^m$, $m=1$ or $2$, on
${\mathcal Z}_{+}\times {\mathcal Z}_{+}$  is integrable if and only for every $J=(J_1,J_2)\in {\mathcal Z}_{+}\times
{\mathcal Z}_{+}$ and every $X,Y,Z,U\in T_pM$, $p=\pi(J)$,
\begin{equation}\label{int-1}
\begin{array}{c}
g({\mathcal R}^D(X\wedge Y-J_1X\wedge J_1Y),Z\wedge U-J_iZ\wedge J_iU)\\[6pt]
\hspace{2cm}=\pm g({\mathcal R}^D(J_1X\wedge Y+X\wedge
J_1Y),J_iZ\wedge U+Z\wedge J_iU),\quad i=1,2,
\end{array}
\end{equation}
where the minus sign corresponds to the case $m=i=2$.
\medskip

By  Lemmas~\ref{Nijen} and \ref{horN}, the latter identity  for
$i=1$ is the integrability condition for the Atiyah-Hitchin-Singer
almost complex structure on ${\mathcal Z}_{+}$.

Now, we show that the identity (\ref{int-1}) with $i=2$ is
equivalent to
\begin{equation}\label{I}
g({\mathcal R}^D(a),b)=0 ~~ \mbox{for every}~~
a,b\in\Lambda^2_{+}TM.
\end{equation}
Suppose that
\begin{equation}\label{II ident}
\begin{array}{c}
g({\mathcal R}^D(X\wedge Y-J_1X\wedge J_1Y),Z\wedge U-J_2Z\wedge J_2U)\\[6pt]
\hspace{2cm}=(-1)^{m+1}g({\mathcal R}^D(J_1X\wedge Y+X\wedge
J_1Y),J_2Z\wedge U+Z\wedge J_2U)
\end{array}
\end{equation}
for every $(J_1,J_2)$ and every $X,Y,Z,U$, $m=1,2$. Let
$J\in{\mathcal Z}_{+}$ be a complex structure on a tangent space
$T_pM$. Applying (\ref{II ident}) with $(J_1,J_2)=(J,J)$ and
$(J_1,J_2)=(J,-J)$,  we see that in the both cases $m=1$ and $m=2$
\begin{equation}\label{II}
g({\mathcal R}^D(X\wedge Y-JX\wedge JY),Z\wedge U-JZ\wedge JU)=0
\end{equation}
for every $X,Y,Z,U\in T_pM$. Let $(E_1,...,E_4)$ be an oriented
orthonormal basis of $T_pM$, and define $s_i=s_i^{+}$ by means of
this basis. Setting  $J=s_1$, $(X,Y)=(E_1,E_3)$, $(Z,U)=(E_1,E_3)$
in (\ref{II}), we get the identity $g(R^D(s_2),s_2)=0$; for
$(Z,U)=(E_1,E_4)$, we obtain $g(R^D(s_2),s_3)=0$. Also, identity
(\ref{II}) with $J=s_2$, $(X,Y)=(E_1,E_3)$, and $(Z,U)=(E_1,E_2)$
implies $g(R^D(s_2),s_1)=0$. Thus,
\begin{equation}\label{RD-s_2}
g(R^D(s_2),s_1)=g(R^D(s_2),s_2)=g(R^D(s_2),s_3)=0.
\end{equation}
Applying the latter identities for the bases $(E_1,E_4,E_2,E_3)$ and
$(E_1,E_3,E_4,E_2)$, we see that
\begin{equation}\label{RD-s1,s3}
\begin{array}{c}
g(R^D(s_1),s_1)=g(R^D(s_1),s_2)=g(R^D(s_1),s_3)=0,\\[6pt]
g(R^D(s_3),s_1)=g(R^D(s_3),s_2)=g(R^D(s_3),s_3)=0.
\end{array}
\end{equation}
Identities (\ref{RD-s_2}) and (\ref{RD-s1,s3}) clearly imply (\ref{I}). This shows that (\ref{I}) follows from (\ref{II
ident}).

In order to show that (\ref{I}) implies (\ref{II ident}), it is
enough to note that if $J\in{\mathcal Z}_{+}$ and $X,Y\in
T_{\pi(J)}M$, then the $2$-vector $X\wedge Y-JX\wedge
JY\in\Lambda^2_{+}T_{\pi(J)}M$. This follows from the easily
verifying fact that $X\wedge Y-JX\wedge JY$ is orthogonal to
$\Lambda^2_{-}T_pM$, $p=\pi(J)$, since $J$ commutes with every
endomorphism of $T_pM$ lying in $\Lambda^2_{-}T_pM$.

The components of the curvature operator ${\mathcal
R}^D:\Lambda^2TM\to\Lambda^2TM$ with respect to the decomposition
$\Lambda^2TM=\Lambda^2_{+}TM\oplus\Lambda^2_{-}TM$ have been
computed in \cite{F}. In fact, using (\ref{curv-D}), (\ref{T-tau})
and (\ref{nabla-T-tau}), we compute
$$
\begin{array}{l}
g({\mathcal R}^D
(s_1),s_1)=g(R^D(E_1,E_2)E_1,E_2)+g(R^D(E_1,E_2)E_3,E_4)\\[6pt]
\hspace{6cm}+g(R^D(E_3,E_4)E_1,E_2)+g(R^D(E_3,E_4)E_3,E_4)\\[8pt]
=-\frac{1}{4}[{\mathcal T}(E_1,E_2,E_3)^2+{\mathcal
T}(E_1,E_2,E_4)^2]-\frac{1}{2}[(\nabla_{E_1}{\mathcal
T})(E_2,E_3,E_4)
-(\nabla_{E_2}{\mathcal T})(E_3,E_4,E_1)]\\[6pt]
-\frac{1}{2}[(\nabla_{E_3}{\mathcal
T})(E_4,E_1,E_2)-(\nabla_{E_4}{\mathcal T})(E_1,E_2,E_3)]
-\frac{1}{4}[{\mathcal T}(E_3,E_4,E_1)^2+{\mathcal T}(E_3,E_4,E_2)^2]\\[6pt]
+g(R^{\nabla}(s_1),s_1)\\[6pt]
=-\frac{1}{4}||\tau||^2-\frac{1}{2}\delta\tau +
g(R^{\nabla}(s_1),s_1).
\end{array}
$$
A similar computation gives
$$
\begin{array}{c}
g({\mathcal R}^D(s_1),s_2)
=\displaystyle{\frac{1}{2}}(d\tau)(s_3)+g({\mathcal
R}^{\nabla}(s_1),s_2).
\end{array}
$$
Also,
$$
g({\mathcal R}^D(s_1),s_3)= -\frac{1}{2}(d\tau)(s_2)+g({\mathcal
R}^{\nabla}(s_1),s_3).
$$
Replacing the basis $(E_1,E_2,E_2,E_4)$ by $(E_1,E_3,E_4,E_2)$ and
$(E_1,E_4,E_2,E_3)$, we see that
\begin{equation}\label{R-D-++}
\begin{array}{c}
\displaystyle{g({\mathcal R}^D
(s_i),s_i)=-\frac{1}{4}||\tau||^2-\frac{1}{2}\delta\tau +
g(R^{\nabla}(s_i),s_i)},\quad i=1,2,3,\\[8pt]
\displaystyle{g({\mathcal
R}^D(s_i),s_j)=\frac{1}{2}(d\tau)(s_i\times s_j)+g({\mathcal
R}^{\nabla}(s_i),s_j)},\quad i\neq j.
\end{array}
\end{equation}

\medskip

Now, suppose that the structure $\mathscr{J}^m$ is integrable, $m=1$
or $2$. Then the manifold $(M,g)$ is anti-self-dual by
Theorem~\ref{Z+}, and identities (\ref{I}),  (\ref{R-D-++}), and
(\ref{dec}) imply
\begin{equation}\label{i+i}
-\frac{1}{4}||\tau||^2-\frac{1}{2}\delta\tau +\frac{s}{6}=0,\quad (d\tau)|\Lambda^2_{+}TM=0.
\end{equation}
The latter identity means that $\ast d\tau=-d\tau$, i.e., the
$2$-form $d\tau$ is anti-self-dual.

 Conversely, suppose that
identities (\ref{i+i}) are satisfied and the metric $g$ is
anti-self-dual. Then (\ref{R-D-++}) implies (\ref{I}), hence the
almost complex structure ${\mathscr J}^m$ is integrable on
${\mathcal Z}_{+}\times{\mathcal Z}_{+}$.

Changing the orientation of $M$, we obtain the integrability
condition for the restriction of ${\mathscr J}^m$ to ${\mathcal
Z}_{-}\times{\mathcal Z}_{-}$.

\end{proof}

\noindent {\bf Remarks} 1. In the case when $\tau=0$, i.e., $D$ is
the Levi-Civita connection, Theorem~\ref{Z++} coincides with
\cite[Theorem 2 (a) and (b)]{Des}

\smallskip

2.  As is well-known, if $M$ is compact, the condition that $d\tau$
is anti-self-dual implies  $d\tau=0$ by integrating the identity
$d(d\tau\wedge\tau)=-d\tau\wedge\star d\tau=-||d\tau||^2vol$.

\smallskip

3. The identity
\begin{equation}\label{conf}
\displaystyle{s=\frac{3}{2}||\tau||^2+3\delta\tau}
\end{equation}
can be interpreted in terms of the Weyl geometry. Recall that a Weyl
connection on a conformal $n$-dimensional manifold is a torsion-free
connection $\nabla^w$ which preserves the conformal structure. This
means that for every Riemannian metric $g$ in the conformal class
there exists a $1$-form $\theta_g$ such that
$\nabla^wg=\theta_g\otimes g$; obviously, such a form is unique. If
$\nabla^g$ is the Levi-Civita connection of the metric $g$,
\begin{equation}\label{Weyl}
\nabla^w_{X}Y=\nabla^g_{X}Y-\frac{1}{2}[\theta_g(X)Y+\theta_g(Y)X-g(X,Y)\theta_g^{\sharp}],
\end{equation}
where $\theta_g^{\sharp}$ is the dual vector field of the form
$\theta_g$ with respect to the metric $g$,
$g(\theta^{\sharp},Z)=\theta_g(Z)$. If $\widetilde{g}=e^{f}g$ is
another Riemannian metric in the conformal class, where $f$ is a
smooth function, $\theta_{\widetilde{g}}=df+\theta_{g}$. Hence the
condition $d\theta_g=0$ does not depend on the choice of the metric
$g$. In this case, we say that $\nabla^w$ determines a closed Weyl
structure. If $d\theta_g=0$, then locally $\theta_g=d\psi$ for a
smooth function $\psi$, so $\nabla^w$ concides locally with the
Levi-Civita connection of the metric $e^{-\psi}g$. The condition
that the form $\theta_g$ is exact also does not depend on the choice
of $g$, and a Weyl structure with exact $\theta_g$ is called exact.

It follows from (\ref{Weyl}) that the Ricci tensors $Ric^w$ and
$Ric^g$ of the connections $\nabla^w$ and $\nabla^g$ are related by
$$
\begin{array}{c}
Ric^w(X,Y)=Ric^g(X,Y)+\displaystyle{\frac{n-1}{2}(\nabla^g_{X}\theta_g)(Y)-\frac{1}{2}(\nabla^g_{Y}\theta_g)(X)}\\[10pt]
\hspace{4cm}-\displaystyle{\frac{n-2}{4}[||\theta_g||^2g(X,Y)-\theta_g(X)\theta_g(Y)]-\frac{1}{2}(\delta^g\theta_g)
g(X,Y)},
\end{array}
$$
where the norm and the codifferential are taken with respect to the
metric $g$. Hence  the symmetric part
$Ric^{sym}=\frac{1}{2}[Ric^w(X,Y)+Ric^w(Y,X)]$ of the Ricci tensor
of the connection $\nabla^w$ is
\begin{equation}\label{Ricci-W}
\begin{array}{c}
Ric^{sym}(X,Y)=Ric^g(X,Y)+\displaystyle{\frac{n-2}{2}}[(\nabla^g_X\theta_g)(Y)+(\nabla^g_Y\theta_g)(X)]\\[8pt]
\hspace{4cm}-\displaystyle{\frac{n-2}{4}[||\theta_g||^2g(X,Y)-\theta_g(X)\theta_g(Y)]-\frac{1}{2}(\delta^g\theta_g)}g(X,Y)].
\end{array}
\end{equation}
Therefore, if $s^g$ is the scalar curvature of $g$, the trace of
$Ric^{sym}$ with respect to $g$ is
\begin{equation}\label{conf-sc}
s_g^{w}=s^g-\frac{(n-1)(n-2)}{4}||\theta_g||^2-(n-1)\delta^g\theta_g.
\end{equation}
Clearly, if $\widetilde{g}=e^{f}g$, then
$s_{\widetilde{g}}^{w}=e^{-f}s_g^w$. The function $s_g^w$ is called
the conformal scalar curvature with respect to $g$. Its vanishing
does not depend on the choice of the metric $g$ in the conformal
class. We note also that  $s_g^w$ represents a section of the
so-called line bundle of weight $-2$, see, for example,
\cite{CalPed,Gaud}.

\smallskip

Now, consider the conformal structure on the Riemannian $4$-
manifold $M$ determined by the metric $g$. Define a connection
$\nabla^w$ on $M$ by (\ref{Weyl}) with $\theta_g=\tau$. Then
$\nabla^w$ is a Weyl connection and the identity (\ref{conf}) is
equivalent to the vanishing of the conformal scalar curvature.

\smallskip

As is well-known, the Weyl tensor is conformally invariant, hence so
are the conditions ${\mathcal W}_{+}=0$ and  ${\mathcal W}_{-}=0$.
Therefore, if $(M,g,\tau)$ satisfies the conditions of
Theorem~\ref{Z++} and $f$ is a smooth function on $M$, then
$(M,e^fg,df+\tau)$ also satisfies these conditions.

As we have noticed, $d\tau=0$ in the case of a compact manifold $M$.
Suppose that $\tau=df$ for a smooth function $f$ on $M$. Let
$\widetilde\delta$ be the codifferential with respect to the metric
$\widetilde{g}=e^{-2f}g$. Using the well-known formula
$\widetilde{\delta}\tau=e^{2f}(\delta\tau+2\tau(grad\, f))$, we see
that if $\tau$ satisfies (\ref{conf}),
$$
e^{2f}s=-\frac{9}{2}e^{2f}||df||^2_g+3\widetilde{\delta}df.
$$
Hence, if $\widetilde\mu$ is the volume form with respect to the
metric $\widetilde{g}$,
$$
\int_{M}e^{2f}s\,\widetilde{\mu}=-\frac{9}{2}\int_M
e^{2f}||df||^2_g\,\widetilde{\mu}\leq 0.
$$
Also, if $\mu$ is the volume form with respect to the metric $g$, it
follows from (\ref{conf}) that
$$
\int_M s\,\mu=\frac{3}{2}\int_M ||\tau||^2_g\mu\geq 0.
$$
Therefore either $s\equiv 0$ or $s$ takes values with different
signs. Thus, if $H^1(M,{\mathbb R})=0$ and $s=const\neq 0$, there
does not exist $1$-form $\tau$ satisfying (\ref{conf}). In
particular, this holds for  $M=S^4$ and $M={\mathbb C}{\mathbb
P}^2$.

\medskip

\noindent {\bf Examples}. 1.  Let $(M,g,J)$ be a Hermitian surface
with fundamental $2$-form $\Omega(X,Y)=g(X,JY)$ and Lee from
$\theta=\delta\Omega\circ J$. Denote by $s$ and $s^{\ast}$ the
scalar and $\ast$-scalar curvatures. Then, by \cite[Theorem
3.1]{Vais},
$$
s-s^{\ast}=||\theta||^2+2\delta\theta.
$$
According to \cite[Proposition 6.4]{Koda}, a Hermitian surface is
anti-self-dual if and only if $s=3s^{\ast}$ and the $2$-form
$d\theta$ is anti-self-dual. Therefore the Lee form of an
anti-self-dual Hermitian surface satisfies the conditions of
Theorem~\ref{Z++} with $\tau=\theta$. Apparently, if the surface is
K\"ahler, $\theta=0$. Note that, as is well-known, a K\"ahler
surface is anti-self-dual if and only if it is scalar-flat. In
particular, by the famous Yau's solution of the Calabi conjecture,
every compact K\"ahler surface with vanishing first Chern class
admits an anti-self-dual K\"ahler metric. There are many other
constructions of scalar-flat K\"ahler surfaces and we refer to
LeBrun's survey paper \cite{LeBrun95} and the literature therein, as
well as to \cite{LeBrun88, KP, RS}, for results in this area. The
standard Hermitian structure on the Hopf surface $S^3\times S^1$ is
non-K\"ahler, conformally flat (${\mathcal W}_{+}={\mathcal
W}_{-}=0$) and locally conformally K\"ahler ($d\theta=0$), so it
satisfies the conditions of Theorem~\ref{Z++}. C. LeBrun
\cite{LeBrun91} has constructed non-K\"ahler anti-self-dual
Hermitian metrics on the blow-ups $(S^3\times S^1)\sharp\,
n\overline{{\mathbb C}{\mathbb P}^2}$ of the Hopf surface. I. Kim
\cite{Kim} has shown the existence of anti-self-dual strictly almost
K\"ahler structures on ${\mathbb C}{\mathbb P}^2\sharp\,
n\overline{{\mathbb C}{\mathbb P}^2}$, $n\geq 11$, $(S^2\times
\Sigma)\sharp\, n\overline{{\mathbb C}{\mathbb P}^2}$, $(S^2\times
T^2)\sharp\, n\overline{{\mathbb C}{\mathbb P}^2}$, $n\geq 6$, where
$\Sigma$ is a compact Riemann surface of genus $\geq 2$ and $T^2$ is
the  torus.

\smallskip

\noindent 2. Following \cite{SS}, consider a real $4$-dimensional
vector space $V$ oriented by a basis $(E_1,...,E_4)$. For
$\lambda\in{\mathbb R}$, let $\frak{g}_{\lambda}$ be the Lie algebra
on $V$ defined by the relations
$$
[E_1,E_2]=E_2-\lambda E_3,\quad [E_1,E_3]=\lambda E_2+E_3,\quad
[E_1,E_4]=2E_4, \quad [E_2,E_3]=-E_4,
$$
and the other brackets equal zero. It is easy to check that this
algebra is solvable.

It is shown in \cite{SS} that $\frak{g}_{\lambda}$ and
$\frak{g}_{\lambda'}$ are not isomorphic if $\lambda'\neq
\pm\lambda$. The algebras $\frak{g}_{\lambda}$ and
$\frak{g}_{-\lambda}$ are isomorphic by the map which  just  changes
the signs of $E_3$ and $E_4$.

Let $G_{\lambda}$ be the simply connected Lie group corresponding to
$\frak{g}_{\lambda}$. This group is not compact. Otherwise, it would
admit an inner product such that all endomorphisms $ad_X$,
$X\in\frak{g}_{\lambda}$, were skew-symmetric. Hence $E_4$ would
have zero length since $ad_{E_1}(E_4)=2E_4$, a contradiction. Note
also that $Trace\,ad_{E_1}\neq 0$, so the group  $G_{\lambda}$ is
not unimodular, hence it does not admit a co-compact lattice.

For $k>0$, let $g_k$ be the left invariant metric on $G_{\lambda}$
for which $(\frac{1}{k}E_1,E_2,E_3,E_4)$ is an orthonormal basis. It
is proved in \cite{SS} that the Riemannian manifolds
$(G_{\lambda},g_k)$ and $(G_{\lambda'},g_{k'})$ are isometric if and
only if $k=k'$. Hence the curvature of $(G_{\lambda},g_k)$ does not
depend on $\lambda$. According to \cite{SS}, the Ricci tensor of
$g_k$ is given by
\begin{equation}\label{Ric-gk}
Ric(E_1,E_1)=6,\quad
Ric(E_2,E_2)=Ric(E_3,E_3)=\frac{4}{k^2}+\frac{1}{2},\quad
Ric(E_4,E_4)=\frac{8}{k^2}-\frac{1}{2},
\end{equation}
and $Ric(E_i,E_j)=0$ for $i\neq j$. Hence the scalar curvature of
$g_k$ is
$$
s_k=\frac{22}{k^2}+\frac{1}{2}.
$$
Also, the two halves of the Weyl operator are computed in \cite{SS}.
The result there shows that ${\mathcal W}_{+}=0$ if and only if
$k\in\{1,2\}$ and, for these values of $k$, ${\mathcal W}_{-}\neq
0$. The orientation of $G_{\lambda}$ is, of course, just a matter of
choice and changing the orientation interchanges the roles of
${\mathcal W}_{+}$ and  ${\mathcal W}_{-}$.

It is proved in \cite{SS} that if $G$ is  an oriented
four-dimensional Lie group admitting a left invariant Riemannian
metric $g$ such that ${\mathcal W}_{+}=0$ and ${\mathcal W}_{-}\neq
0$, then the Lie algebra of $G$ is isomorphic to
$\frak{g}_{\lambda}$ for some $\lambda\geq 0$ and $g$ is locally
homothetic to either $g_1$ or $g_2$.

Extend $(E_1,...,E_4)$ to a frame of left-invariant vector fields on
$G_{\lambda}$.  Let $(\alpha^1,...,\alpha^4)$ be the dual frame to
$(E_1,...,E_4)$. Then
$$
\begin{array}{c}
d\alpha^1=0,\quad
d\alpha^2=-\alpha^1\wedge\alpha^2-\lambda\alpha^1\wedge\alpha^3,
\\[6pt]
d\alpha^3=\lambda\alpha^1\wedge\alpha^2-\alpha^1\wedge\alpha^3,
\quad d\alpha^4=-2\alpha^1\wedge\alpha^4+\alpha^2\wedge\alpha^3.
\end{array}
$$

Let $\tau=\sum_{i=1}^4\mu_i\alpha^i$, $\mu_i\in{\mathbb R}$,   be a
left invariant $1$-form on $G_{\lambda}$. Then, if $\delta^k$ is the
codifferential with respect to the metric $g_k$,
$\delta^k\tau=\frac{4}{k^2}\mu_1$. Let $s_1^{+},s_2^{+},s_3^{+}$ be
the basis of $\Lambda^2_{+}\frak{g}_{\lambda}$ defined via
(\ref{s-basis}) by means of the orthonormal basis
$(\frac{1}{k}E_1,E_2,E_3,E_4)$, where $k=1$ or $k=2$. The identities
$d\tau(s_i^{+})=0$ for $i=1,2,3$ are equivalent to
$\mu_2=\mu_3=\mu_4$ in the case $k=1$ and $\mu_2=\mu_3=0$ when
$k=2$.   Thus, if $k=1$, identity (\ref{conf}) is satisfied for
$\mu_1=-4\pm\sqrt{31}$; if $k=2$, it is satisfied iff  $\mu_1$ and
$\mu_4$ are related by $\mu_1^2+8\mu_1-16+4\mu_4^2=0$.

\smallskip

\begin{theorem}\label{+-}
The almost complex structure $\mathscr{J}^m$ on ${\mathcal
Z}_{+}\times {\mathcal Z}_{-}$ (resp. ${\mathcal Z}_{-}\times
{\mathcal Z}_{+}$), $m=1$ or $2$,  is integrable if and only if the
metric $g$ is anti-self-dual (resp. self-dual) and its Ricci tensor
$\rho$ is given by
$$
\rho(X,Y)=\big[{\mathscr
S}(\nabla\tau)-\frac{1}{2}\tau\otimes\tau+\frac{1}{8}(2s+2\delta\tau+||\tau||^2)g\big](X,Y),
$$
where ${\mathscr S}$ stands for the symmetrization of a bilinear
form.
\end{theorem}

\begin{proof}
By Proposition~\ref{Ni} and Lemma~\ref{horN},  the almost complex structure $\mathscr{J}^m$ on ${\mathcal Z}_{+}\times
{\mathcal Z}_{-}$  is integrable if and only the metric $g$ is anti-self-dual and for every $J=(J_1,J_2)\in {\mathcal
Z}_{+}\times {\mathcal Z}_{-}$
\begin{equation}\label{int}
\begin{array}{c}
g({\mathcal R}^D(X\wedge Y-J_1X\wedge J_1Y),Z\wedge U-J_2Z\wedge J_2U)\\[6pt]
\hspace{2cm}=(-1)^{m+1} g({\mathcal R}^D(J_1X\wedge Y+X\wedge J_1Y),J_2Z\wedge U+Z\wedge J_2U)
\end{array}
\end{equation}
for every $X,Y,Z,U\in T_{\pi(J)}M$, $m=1,2$.

 Let $(E_1,...,E_4)$ be an oriented orthonormal basis
of a tangent space $T_pM$, and define $\{s_i^{\pm}\}$ via
(\ref{s-basis}). Then identity (\ref{int}) is equivalent to
\begin{equation}\label{intij} g(R^D(s_i^{+}),s_j^{-})=0,\quad
i,j=1,2,3.
\end{equation}
Indeed, if the latter identity holds, the identity (\ref{int}) holds as well since $A\wedge B-J_1A\wedge J_1B\in
\Lambda^2_{+}T_{\pi(J)}M$, $A\wedge B-J_2A\wedge J_2B\in \Lambda^2_{-}T_{\pi(J)}M$ for every $A,B\in T_{\pi(J)}M$.
Conversely, suppose that (\ref{int}) is satisfied. Then this identity with $J_1=s_1^{+}$, $J_2=(-1)^{m+1}s_1^{-}$,
$(X,Y)=(Z,U)=(E_1,E_3)$ gives
$$
g(R^D(s_2^{+}),s_2^{-})=-g(R^D(s_3^{+}),s_3^{-}).
$$
Replacing the basis $(E_1,E_2,E_3,E_4)$ by $(E_1,E_3,E_4,E_2)$ and
$(E_1,E_4,E_2,E_3)$, we get from the identity above
$$
g(R^D(s_3^{+}),s_3^{-})=-g(R^D(s_1^{+}),s_1^{-}),\quad g(R^D(s_1^{+}),s_1^{-})=-g(R^D(s_2^{+}),s_2^{-}).
$$
It follows that
\begin{equation}\label{ii}
g(R^D(s_i^{+}),s_i^{-})=0,\quad i=1,2,3.
\end{equation}
Identity (\ref{int}) with $J_1=s_1^{+}$, $(X,Y)=(E_1,E_4)$, ~
$J_2=(-1)^{m+1}s_3^{-}$, $(Z,U)=(E_1,E_2)$ implies
$g(R^D(s_3^{+}),s_1^{-})=g(R^D(s_2^{+}),s_2^{-})$. Setting
$J_1=s_2^{+}$, $(X,Y)=(E_1,E_4)$, $J_2=(-1)^{m+1}s_3^{-}$,
$(Z,U)=(E_1,E_3)$,  we also get
$g(R^D(s_3^{+}),s_2^{-})=g(R^D(s_1^{+}),s_1^{-})$. Thus,
$$
g(R^D(s_3^{+}),s_1^{-})=g(R^D(s_3^{+}),s_2^{-})=0.
$$
Again replacing  the basis $(E_1,E_2,E_3,E_4)$ by the bases
$(E_1,E_3,E_4,E_2)$ and \\$(E_1,E_4,E_2,E_3)$, we see that
\begin{equation}\label{ij}
g(R^D(s_i^{+}),s_j^{-})=0, \quad i\neq j.
\end{equation}
\smallskip
Now, using (\ref{curv-D}), we compute
$$
\begin{array}{c}
g({\mathcal
R}^D(s_1^{+}),s_1^{-})=\frac{1}{2}\big[(\nabla_{E_1}{\mathcal
T})(E_{234})-(\nabla_{E_2}{\mathcal
T})(E_{134})-(\nabla_{E_3}{\mathcal
T})(E_{124}))+(\nabla_{E_4}{\mathcal T})(E_{123})\big]\\[8pt]
+\frac{1}{4}\big[-{\mathcal T}_{123}^2-{\mathcal T}_{124}^2
+{\mathcal T}_{134}^2 +{\mathcal T}_{234}^2\big]+g({\mathcal
R}^{\nabla}(s_1^{+}),s_1^{-}).
\end{array}
$$
Moreover, setting $\rho_{ij}=\rho(E_i,E_j)$, we have by (\ref{B})
\begin{equation}\label{B11}
g({\mathcal R}^{\nabla}(s_1^{+}),s_1^{-})=g({\mathcal
B}(s_1^{+}),s_1^{-})=\frac{1}{2}\big[\rho_{11}+\rho_{22}
-\rho_{33}-\rho_{44}\big].
\end{equation}
Thus, taking into account (\ref{nabla-T-tau}), we obtain
\begin{equation}\label{11}
\begin{array}{c}
g({\mathcal
R}^D(s_1^{+}),s_1^{-})=\frac{1}{2}\big[-(\nabla_{E_1}{\tau})(E_{1})-(\nabla_{E_2}{\tau})(E_{2})
+(\nabla_{E_3}{\tau})(E_{3}))+(\nabla_{E_4}{\tau})(E_{4})\big]\\[8pt]
+\frac{1}{4}\big[-{\tau}_{4}^2-{\tau}_{3}^2 +{\tau}_{2}^2
+{\tau}_{1}^2\big]+\frac{1}{2}\big[\rho_{11}+\rho_{22}
-\rho_{33}-\rho_{44}\big],
\end{array}
\end{equation}
where $\tau_i=\tau(E_i)$, $i=1,..,4$. Applying the latter identity
for the bases $(E_1,E_3,E_4,E_2)$ and $(E_1,E_4,E_2,E_3)$, we get
$$
0=4\sum\limits_{i=1}^3g({\mathcal
R}^D(s_i^{+}),s_i^{-})=2\big[-4(\nabla_{E_1}\tau)(E_1)-\delta\tau\big]-||\tau||^2+4\tau_{1}^2+
8\rho_{11}-2s.
$$
For a bilinear form $\alpha$, let $({\mathscr
S}\alpha)(X,Y)=\frac{1}{2}[\alpha(X,Y)+\alpha(Y,X)]$ be the
symmetrization of $\alpha$. Then the latter identity can be written
as
\begin{equation}\label{rhoii}
\rho(E_1,E_1)=\big[{\mathscr
S}(\nabla\tau)-\frac{1}{2}\tau\otimes\tau+\frac{1}{8}(2s+2\delta\tau+||\tau||^2)g\big](E_1,E_1).
\end{equation}
Moreover,
$$
\begin{array}{c}
0=g({\mathcal R}^D(s_1^{+}),s_2^{-})+g({\mathcal
R}^D(s_2^{+}),s_1^{-})=(\nabla_{E_2}{\mathcal
T})(E_{124})-(\nabla_{E_3}{\mathcal T})(E_{134})
-{\mathcal T}_{124}{\mathcal T}_{134}\\[8pt]
+g({\mathcal B}(s_1^{+}),s_2^{-})+g({\mathcal B}(s_2^{+}),s_1^{-})\\[8pt]
=
-(\nabla_{E_2}\tau)(E_3)-(\nabla_{E_3}\tau)(E_3)+\tau(E_3)\tau(E_2)+2\rho_{23}.
\end{array}
$$
Therefore
\begin{equation}\label{rhoij}
\rho(E_2,E_3)=\big[{\mathscr
S}(\nabla\tau)-\frac{1}{2}\tau\otimes\tau\big](E_2,E_3).
\end{equation}
If follows from (\ref{rhoii}) and (\ref{rhoij}) that
$$
\rho(E_i,E_j)=\big[{\mathscr
S}(\nabla\tau)-\frac{1}{2}\tau\otimes\tau+\frac{1}{8}(2s+2\delta\tau+||\tau||^2)g\big](E_i,E_j),\quad
i,j=1,...,4.
$$
Therefore, for every $X,Y\in TM$,
\begin{equation}\label{rho}
\rho(X,Y)=\big[{\mathscr
S}(\nabla\tau)-\frac{1}{2}\tau\otimes\tau+\frac{1}{8}(2s+2\delta\tau+||\tau||^2)g\big](X,Y).
\end{equation}

Conversely, suppose that the Ricci tensor of $g$ is given by
(\ref{rho}). Then it follows from (\ref{rhoii}) that $g({\mathcal
R}^D(s_1^{+}),s_1^{-})=0$. An easy computation making use of
(\ref{curv-D}), (\ref{nabla-T-tau}), (\ref{B}), and (\ref{rhoij})
gives  $g({\mathcal R}^D(s_1^{+}),s_2^{-})=0$ and $g({\mathcal
R}^D(s_1^{+}),s_3^{-})=0$. Replacing the basis $(E_1,E_2,E_3,E_4)$
by $(E_1,E_3,E_4,E_2)$ and $(E_1,E_4,E_2,E_3)$, we see that
$g({\mathcal R}^D(s_i^{+}),s_j^{-})=0$ for every $i,j=1,2,3$.
Therefore if $(M,g)$ is anti-self-dual and its Ricci tensor is given
by (\ref{rho}),  the almost complex structure
$\mathscr{J}^m|{\mathcal Z}_{+}\times {\mathcal Z}_{-}$  is
integrable.
\end{proof}

\noindent {\bf Remarks} 1. In the case when $\tau=0$, i.e., $D$ is
the Levi-Civita connection, Theorem~\ref{+-} coincides with
\cite[Theorem 2 (c) and (d)]{Des}

\smallskip

2. The identity
\begin{equation}\label{WE}
\rho(X,Y)=\big[{\mathscr
S}(\nabla\tau)-\frac{1}{2}\tau\otimes\tau+\frac{1}{8}(2s+2\delta\tau+||\tau||^2)g\big](X,Y)
\end{equation}
can also be interpreted in terms of the Weyl geometry. Recall that a
 $n$-dimensional conformal manifold with a Weyl connection $\nabla^ww$ is said to be Einstein-Weyl if the
symmetric part $Ric^{sym}$  of the Ricci tensor of $\nabla^w$ is
proportional to one (hence to every) metric $g$ in the conformal
class. The proportionality factor is clearly $s^w_g/n$, where
$s^w_g$ is the conformal scalar curvature with respect to $g$.

Consider the conformal structure on the Riemannian $4$- manifold $M$
determined by the metric $g$,  and define a connection $\nabla^w$ on
$M$ by (\ref{Weyl}) with $\theta_g=-\tau$. Then it follows from
(\ref{Ricci-W}) and (\ref{conf-sc}) that identity (\ref{WE}) is
equivalent to $\nabla^w$ being Einstein-Weyl.  Hence, if
$(M,g,\tau)$ satisfies the conditions of Theorem~\ref{+-} and $f$ is
a smooth function on $M$, then $(M,e^fg,-df+\tau)$ also satisfies
these conditions.

\smallskip

It has been proved by P. Gauduchon \cite{Gaud} and H. Pedersen - A.
Swann \cite{PedSwa-PLMS} that an Einstein-Weyl structure on a
compact self-dual manifold is closed. Thus, we have the following.

\begin{cor}\label{compact}
If the base manifold $M$ is compact and the almost complex structure
$\mathscr{J}^m$ on ${\mathcal Z}_{+}\times {\mathcal Z}_{-}$ (resp.
${\mathcal Z}_{-}\times {\mathcal Z}_{+}$), $m=1$ or $2$,  is
integrable, then $d\tau=0$.
\end{cor}

\medskip

\noindent {\bf Examples}. 1. Combining the above mentioned result by
P. Gauduchon and H. Pedersen - A. Swann with results due to H.
Pedersen - A. Swan in \cite{PedSwa-Crelle}, D. Calderbank and H.
Pedersen have given the following classification of compact
self-dual Einstein-Weyl manifold \cite[Theorem 9.8]{CalPed}: A
compact self-dual Einstein-Weyl $4$-manifold is isometric to $S^4$,
${\mathbb C}{\mathbb P}^2$, or an Einstein manifold of negative
scalar curvature, or is covered by a flat torus, a $K3$-surface or a
coordinate quaternionic Hopf surface. Thus, these manifolds satisfy
the conditions of Theorem~\ref{+-}.  We refer to \cite{CalPed} and
the literature therein for details.

\smallskip

\noindent 2. Consider the Lie group $G_{\lambda}$ with the (anti-)
self-dual metric $g_k$, $k=1,2$, defined above. Then a
left-invariant $1$-form $\tau$ satisfies identity (\ref{WE}) if and
only if $\tau=0$ and $k=2$ (i.e., the metric $g_2$ is Einstein).

\smallskip

\noindent 3. Other non-compact examples of Einstein-Weyl manifolds
can be found in \cite{Bonn} (see also \cite{CalPed}).

\smallskip

\centerline{{\bf Acknowledgements}}

 The author would like to thank V. Apostolov
and A. Swan for helpful discussions and suggestions. He is also
grateful to the referee whose remarks  helped to improve the final
version of the paper.

\end{document}